\documentclass[11pt]{article}

\usepackage{amsmath}
\usepackage{amssymb}
\usepackage{amsthm}
\usepackage{multirow}
\usepackage{color}
\usepackage{longtable}
\usepackage{array}
\usepackage{url}
\usepackage{enumerate}
\usepackage{subcaption}

\newtheorem{theorem}{Theorem}[section]

\newtheorem{lemma}[theorem]{Lemma}
\newtheorem{example}[theorem]{Example}

\newtheorem{remark}{Remark}[section]

\begin{document}

\title{Experimental Constructions of Binary Matrices with Good Peak-Sidelobe Distances}

\author{
Jerod Michel 
\thanks{J. Michel is with the Department of Mathematics, Zhejiang University, Hangzhou 310027, China (e-mail: contextolibre@gmail.com).}
}
\maketitle
\begin{abstract}
Skirlo et al., in ``Binary matrices of optimal autocorrelations as alignment marks'' [{\it Journal of Vacuum Science and Technology} Series B 33(2) (2015) 1-7], defined a new class of binary matrices by maximizing the peak-sidelobe distances in the aperiodic autocorrelations and, by exhaustive computer searches, found the optimal square matrices of dimension up to $7\times7$, and optimal diagonally symmetric matrices of dimensions $8\times8$ and $9\times9$. We make an initial investigation into and propose a strategy for (deterministically) constructing binary matrices with good peak-sidelobe distances. We construct several classes of these and compare their distances to those of the optimal matrices found by Skirlo et al. Our constructions produce matrices that are near optimal for small dimension. Furthermore, we formulate a tight upper bound on the peak-sidelobe distance of a certain class of circulant matrices. Interestingly, binary matrices corresponding to certain difference sets and almost difference sets have peak-sidelobe distances meeting this upper bound.
\smallskip

\end{abstract}

\section{Introduction}\label{sec1}
Sequences and matrices with good autocorrelation properties have important applications in digital communications such as radar, sonar, code-division multiple access (CDMA), and cryptography \cite{BARK}, \cite{NH}, as well as in coded aperture imaging \cite{GF}.
A less developed problem in matrix design was recently considered by Skirlo et al. in \cite{SKIR}. Here, the application of binary matrices with good aperiodic autocorrelation properties to two-dimensional spatial alignment is noted, where an alignment mark is made by creating a surface pattern different from the background thereby allowing pattern information to transform into a two-level signal while a digital image is taken. Position marks for electron-beam lithography based on such binary matrices, for example, were shown to be immune to noise and certain manufacturing errors in \cite{BK}, but the elements of the application have yet to be optimized. So far, such matrices have been obtained only by exhaustive computer searches.
\par
Let $I$ and $J$ be index sets with $\left| I \right|=M$ and $\left| J \right|=N$. Let $R=(R_{i,j})$ be an $M$ by $N$ binary matrix whose rows and columns are labelled by $I$ and $J$ respectively. The 2D aperiodic autocorrelation $A_{R}(\tau_{1},\tau_{2})$ at integer shifts $\tau_{1}$ and $\tau_{2}$ of the binary matrix $R$ is given by \begin{equation}\label{eq1}
A_{R}(\tau_{1},\tau_{2})=\sum_{i \in I} \sum_{j \in J} R_{i,j} R_{i+\tau_{1},j+\tau_{2}}. \end{equation}
\par
Note that $(A_{R}(\tau_{1},\tau_{2}))$ is an inversion-symmetric $(2M-1)\times(2N-1)$ matrix, i.e., $A_{R}(\tau_{1},\tau_{2})=A_{R}(-\tau_{1},-\tau_{2})$ for all $\tau_{1},\tau_{2}$. Also note that all the matrices are implicitly
padded with 0s for all the matrix elements  of indices exceeding
their matrix dimensions. The crosscorrelation between the matrix $R$ and the data image matrix $R'$ is given by \[ C_{RR'}(\tau_{1},\tau_{2})=\sum_{i=1}^{M}\sum_{j=1}^{N} R_{i,j}R'_{i+\tau_{1},j+\tau_{2}}. \] If the data matrix $R'$ is a noisy version of the reference matrix $R$, the peak value of the crosscorrelation can determine the most probable position of the mark.
\par
The aperiodic autocorrelation of a binary matrix can be expressed as $\{d_{1} | n_{1},n_{2},...,n_{s+1}\}$. Here $d_{1} = l-s$ where $l$ is the number of $1$s in the binary matrix, called the {\it peak}, and $s$ is the highest value of $A_{R}(\tau_{1},\tau_{2})$ for $\tau_{1},\tau_{2}$ not both zero, called the {\it nearest sidelobe}. The other distances are given by $d_{i+1}=d_{i}+1$ for $i\geq 1$, where $n_{i}$ gives the number of times $d_{i}$ occurs in the autocorrelation.
\subsection{Previous Work}
In \cite{SKIR}, two criteria are given for a binary matrix to be optimal; one is that the probability of misalignment, which depends on the values of the sidelobes of the autocorrelation, should be minimized, and the other is that the misalignment deviation, which depends on the positions of the sidelobes relative to the peak value $A_{R}(0,0)$, should be minimized. Therefore, binary matrices with an autocorrelation $\{d_{1} | n_{1},n_{2},...,n_{s+1}\}$ in which \begin{enumerate} \item $d_{1}$ is maximized, and
\item the $n_{i}$s are minimized sequentially in the dictionary order,
\end{enumerate} are desirable. Matrices of any sizes can be compared using these criteria.
\par
In \cite{SKIR}, two upper bounds which depend on the number of 1s in the matrix were computed. Skirlo et al. found that the largest possible $d_{1}$ for all matrices with given $l$ (number of $1$s in the matrix) and fixed dimension, is $d_{1,max}=l-s_{min}(l)$ where $s_{min}(l)$ is the minimum highest sidelobe value as a function of $l$. The first upper bound, $d_{1,max}^{upper,I}(l)$ was computed by maximizing $l-A_{R}(\pm 1,0)$. By computing $A_{R}(\pm 1,0)$ a lower bound $s_{min}^{lower,I}(l)$ on $s_{min}(l)$ is obtained. Assuming $M \leq N$ the first upper bound on $d_{1,max}(l)$ was found to be \begin{equation}\label{eq2}
 d_{1,max}^{upper,I}(l)= \begin{cases}
l,       & l \in \lbrack 0,N_{1} \rbrack     \\
N_{1},   & l \in \lbrack N_{1},N_{2} \rbrack \\
M(N+1)-l,& l \in \lbrack N_{2},MN \rbrack    \\
\end{cases}
\end{equation} where $N_{1}=\frac{MN}{2}$, $N_{2}=\frac{MN}{2}+M$ when $MN$ is even and $N_{1}=\frac{MN+1}{2}$, $N_{2}=\frac{MN+1}{2}+M-1$ if $MN$ is odd. The other upper bound that was  found in \cite{SKIR} has similar dependencies. The drawbacks of these upper bounds are that they are not tight, which can easily be seen by checking any of the peak-sidelobe distances of the optimal matrices they found against the bound, and that they are functions of $l$ and not just $M$ and $N$. The aperiodicity of the autocorrelation function makes the problem of formulating a tight upper bound that depends only on the dimension of the binary matrix difficult.
\subsection{Overview of Proposed Strategy}\label{ssec1.1}
  Here we propose a strategy for an initial investigation into the problem of constructing binary matrices with good peak-sidelobe distances. Concerning the $M\times M$ optimal matrices found in \cite{SKIR}, three important observations can be made: the first is that most of them are diagonally-symmetric, the second is that the nearest sidelobe always occurs within $\{A_{R}(\pm 1,0),A_{R}(0,\pm 1)\}$, and the third is that they all consist of an {\it interior} ($(M-2)\times(M-2)$ matrix obtained by deleting the first and last row and the first and last column) in which each row and column has roughly the same number of 1s, and an {\it exterior} (first and last row, and first and last column) in which every entry is a $1$ except for possibly one $0$ on each side. These observations are illustrated in Figure \ref{fig1} by the optimal binary matrices of dimensions $6\times 6$ resp. $7\times 7$ that were found in \cite{SKIR}.
\begin{figure}\begin{center}\begin{subfigure}[b]{0.33\textwidth} $ \left[\begin{array}{cccccc}
1&1&0&1&1&1\\
1&0&1&0&0&1\\
0&1&1&0&1&0\\
1&0&0&0&1&1\\
1&0&1&1&0&1\\
1&1&0&1&1&1\end{array}\right]$
\end{subfigure}
\begin{subfigure}[b]{0.33\textwidth}
$\left[\begin{array}{ccccccc}
1&1&1&1&0&1&1\\
1&0&0&1&1&0&1\\
0&1&1&0&1&0&1\\
1&1&0&1&0&1&1\\
1&0&1&0&0&1&0\\
1&0&0&1&1&0&1\\
1&1&1&0&1&1&1\end{array}\right]$
\end{subfigure}\caption{\footnotesize{$6\times6$ and $7\times7$ optimal binary matrices}}\label{fig1}\end{center}
\end{figure}
\par
This first and third observations lead us to consider the following general method for constructing a binary $M\times M$ matrix with a good peak-sidelobe distance:
\\
\\
{\bf Step 1}: Construct an $(M-2)\times(M-2)$ circulant matrix whose peak-sidelobe distance is large among the class of all $(M-2)\times(M-2)$ circulant matrices.
\\
\\
{\bf Step 2}: Obtain an $M\times M$ matrix by giving a border of $1$s to the matrix obtained in Step 1.
\\
\\
{\bf Step 3}: Try to increase the peak-sidelobe distance of the matrix obtained in Step 2 by changing some of the $1$s on the border to $0$s.
\\
\\
Since computing the peak-sidelobe distance at an arbitrary shift is troublesome due to the aperiodicity of the autocorrelation, the second observation mentioned above suggests considering, as the interior matrix, circulant matrices whose nearest sidelobe occurs within $\{A_{R}(\pm 1,0),A_{R}(0,\pm 1)\}$. This ensures that the peak-sidelobe distance will be $A_{R}(0,0)-A_{R}(\tau_{1}',\tau_{2}')$ where $A_{R}(\tau_{1}',\tau_{2}')\in \{A_{R}(\pm 1,0),A_{R}(0,\pm 1)\}$. It turns out that if we assume the interior matrix is a circulant matrix such that in each row (or column), the number of pairs of consecutive $1$s is large enough, the nearest sidelobe will occur in $\{A_{R}(\pm 1,0),A_{R}(0,\pm 1)\}$. We formulate a tight upper bound for the peak-sidelobe distance of such a circulant matrix and, interestingly, circulant matrices corresponding to certain difference sets and almost difference sets have peak-sidelobe distances meeting this upper bound. It turns out that we can find $(M-2)\times(M-2)$ circulant matrices whose peak-sidelobe distances are optimal among the class of all such $(M-2)\times(M-2)$ circulant matrices. Using such matrices as the interior matrix, after adjoining an exterior of $1$s, and then carefully choosing which of these $1$s to change to $0$s, we obtain $M\times M$ binary matrices with peak-sidelobe distances that are good in the sense that they are very close to being optimal for small dimensions. We construct several families of such matrices and, for small dimensions, compare their peak-sidelobe distances to those of the optimal matrices found in \cite{SKIR}.

\par
In Section \ref{sec2} we discuss some necessary preliminary concepts. In Section \ref{sec3} we formulate a tight upper bound for the peak-sidelobe distance of a certain class of binary circulant matrices. In Section \ref{sec4} we construct several binary circulant matrices with peak-sidelobe distances meeting the upper bound. In Section \ref{sec5} we construct several families of binary matrices with good peak-sidelobe distances and, for small dimension, compare these to the peak-sidelobe distances of the optimal matrices found in \cite{SKIR}. Section \ref{sec6} concludes the paper.

\section{Preliminaries}\label{sec2}
\subsection{Periodic Distance, Difference Sets and Circulant Matrices}
Let $G$ be an additive group of order $v$, and $k$ a positive integer such that $2\leq k < v$. A $k$-element subset $D \subseteq G$ has {\it difference levels} $\mu_{1} < \cdots < \mu_{s}=:\Lambda$ if there exist integers $t_{1},...,t_{s}$ such that the multiset \[ M=\{g-h \mid g,h \in D\} \] contains exactly $t_{i}$ members of $G-\{0\}$ each with multiplicity $\mu_{i}$ for all $i$, $1 \leq i \leq s$. The {\it periodic distance} of $D$, denoted $d(D)$, is defined by $d(D)=k-\Lambda$. We say that $D$ is {\it cyclic} if $G$ is cyclic. If $D$ is cyclic and the number of pairs of consecutive residues in $D$ is $\Lambda$, then we say $D$ is {\it special}. In the case where $s=1$, $D$ is called a $(v,k,\Lambda)$ {\it difference set} \cite{HALL}, and in the case where $s=2$ and $\Lambda=\mu_{2}=\mu_{1}+1$, $D$ is called a $(v,k,\mu_{1},t_{1})$ {\it almost difference set} \cite{NOW}.
\begin{theorem}\label{th23} {\rm \cite{STIN}} The set $D$ is a $(v,k,\lambda)$ difference set in Abelian group $G$ if and only if its complement $D^{c}$ is a $(v,v-k,v-2k+\lambda)$ difference set in $G$.
\end{theorem}
\begin{theorem}\label{th24} {\rm \cite{ACH}} The set $D$ is a $(v,k,\lambda,t)$ almost difference set in Abelian group $G$ if and only if its complement $D^{c}$ is $(v,v-k,v-2k+\lambda,t)$ almost difference set in $G$.
\end{theorem}
\par
We call the set $\{D+g \mid g \in G\}$ of translates of $D$, denoted by $Dev(D)$, the {\it development} of $D$. Let $R=(R_{g,h})$ be the $v\times v$ matrix defined by \[ R_{g,h}=\begin{cases} 1, \text{ if } g\in D+h,\\
                                         0, \text{ otherwise,}\end{cases}\]
for $g,h\in G$. Then we say $R$ is the {\it incidence matrix} of $Dev(D)$. If $G$ is cyclic then we say $R$ is a {\it binary circulant matrix} with defining set $D$, and we say that a $R$ is {\it special} if $D$ is special.

\subsection{Group Ring Notation}
It is sometimes convenient to represent binary matrices by members of group rings. Let $G$ be an additive Abelian group and $\mathbb{Z}$ the ring of integers. Define the group ring $\mathbb{Z} \left[ G \right]$ to be the ring of all formal sums \[
\mathbb{Z} \left[ G \right]=\left\{\sum_{g \in G}a_{g}X^{g} \mid a_{g} \in \mathbb{Z} \right\} \]where $X$ is an indeterminate. The ring $\mathbb{Z}\left[ G \right]$ has the operation of addition given by \[
\sum_{g \in G}a_{g}X^{g}+\sum_{g \in G}b_{g}X^{g}=\sum_{g \in G}(a_{g}+b_{g})X^{g},\]
and the operation of multiplication defined by\[
\left(\sum_{g \in G}a_{g}X^{g}\right)\left(\sum_{g \in G}b_{g}X^{g}\right)=\sum_{h\in G}\left(\sum_{g\in G}a_{g}b_{h-g}\right)X^{h}.\]
We denote the unit of $\mathbb{Z} \left[ G \right]$ by $X^{\mathbf{0}}=\mathbf{1}$ where $\mathbf{0}$ is an additive identity.
\subsection{Cyclotomic Classes and Cyclotomic Numbers}
Let $q$ be a prime power, and $\gamma$ a primitive element of $\mathbb{F}_{q}$. The {\it cyclotomic classes} of order $e$ are given by $C_{i}^{e}=\gamma^{i}\langle \gamma^{e} \rangle$ for $i=0,1,...,e-1$. Define $(i,j)_{e}=|C_{i}^{e}\cap (C_{j}^{e}+1)|$. It is easy to see there are at most $e^{2}$ different cyclotomic numbers of order $e$. When it is clear from the context, we simply denote $(i,j)_{e}$ by $(i,j)$.
The cyclotomic numbers $(h,k)$ of order $e$ have the following properties \cite{D}:

\begin{eqnarray}\label{eq7}
(h,k) & = & (e-h,k-h), \\
(h,k) & = & \begin{cases}
(k,h),                         & \text{if } f \text{ even},\\
(k+\frac{e}{2},h+\frac{e}{2}), & \text{if } f \text{ odd}.
\end{cases}
\end{eqnarray}

We will also need the following lemmas.
\begin{lemma}\label{le5} {\rm \cite{NOW}} \[ -D_{i}^{e} := \{-x \mid x \in D_{i}^{e} \} = \begin{cases}
D_{i}^{e}, & \text{if } f \text{ is even} \\
D_{i+\frac{e}{2}}^{e}, & \text{if } f \text{ is odd}.
\end{cases} \]
\end{lemma}
\begin{lemma}\label{le2.0} {\rm \cite{NOW}} Let $q=em+1$ be a prime power for some positive integers $e$ and $f$. In the group ring $\mathbb{Z}\left[\mathbb{F}_{q}\right]$ we have \[
C_{i}^{e}(X)C_{j}^{e}(X^{-1})=a_{ij}\mathbf{1}+\sum_{k=0}^{e-1}(j-i,k-i)_{e}C_{k}^{e}(X)\] where \[
a_{ij}=\begin{cases} f, \text{ if } m \text{ is even and } j=i, \\
                     f, \text{ if } m \text{ is odd and } j=i+\frac{e}{2}, \\
                     0, \text{ otherwise. }\end{cases}\]
\end{lemma}
\section{A General Upper Bound on the Peak-Sidelobe Distance of a Special Matrix}\label{sec3}

 We first formulate the following general upper bound on the periodic distance of a subset of a cyclic group.

 \begin{lemma}\label{le11} Let $D$ be a $k$-subset of $\mathbb{Z}_{v}$ with periodic distance $d$. Then \[ d \leq \left\lfloor \frac{v^{2}}{4(v-1)} \right\rfloor. \]
 \end{lemma}
 \begin{proof}
  Suppose $D$ has difference levels $\mu_{1}<\cdots<\mu_{s}=\Lambda$, and let $t_{i}$ denote the number of members of the multiset $S=\{g-g'\mid g,g' \in D \text{ and } g \neq g'\}$ with multiplicity $\mu_{i}$. Then we have $d=k-\Lambda$ and, counting $|S|$ in two ways, we have \[
 (\Lambda-\delta_{s-1})t_{s-1}+(\Lambda-\delta_{s-2})t_{s-2}+\cdots+(\Lambda-\delta_{1})t_{1}+\Lambda (v-1-t_{1}-\cdots -t_{s-1})=k(k-1) \]
 where $\delta_{i}=\Lambda-\mu_{i}$. Notice we must have $1\leq t_{i} \leq v$ for all $i$, $1\leq i\leq s$, and all $\delta_{i}$s must be nonnegative and distinct. Set $a=v-k$. Then we have \[ \Lambda=\frac{k(k-1)+\delta_{1}t_{1}+\cdots+\delta_{s-1}t_{s-1}}{v-1} \] whence \begin{eqnarray*}
 k-\Lambda& = & v-a-\frac{(v-a)(v-a-1)}{v-1}-\frac{\delta_{1}t_{1}+\cdots+\delta_{s-1}t_{s-1}}{v-1} \\
          & = & (v-a)\frac{a}{v-1}-\frac{\delta_{1}t_{1}+\cdots+\delta_{s-1}t_{s-1}}{v-1}. \end{eqnarray*}
 By differentiating the first term of the right hand side with respect to $a$ we find that it attains its maximum value when $a=\frac{v}{2}$. Thus we get \begin{eqnarray*}
 k-\Lambda & \leq & \frac{v^{2}}{4(v-1)}.
 \end{eqnarray*}
 \end{proof}

  Again let $D$ be a $k$-subset of $\mathbb{Z}_{v}$ with periodic distance $d$ and difference levels $\mu_{1}<\cdots<\mu_{s}=\Lambda$. Let $a,\delta_{i}$ and $t_{i}$ be defined as in the proof of Lemma \ref{le11}. Suppose that
 \begin{equation} \label{eq13}
\left|k-\frac{v}{2}\right|^{2}+\sum_{i=1}^{s-1}\delta_{i}t_{i} = \frac{v^{2}}{4}-(v-1)\left\lfloor \frac{v^{2}}{4(v-1)} \right\rfloor.
 \end{equation}
Then from the proof of Lemma \ref{le11} we have \begin{eqnarray*}
 k-\Lambda & = & (v-a)\frac{a}{v-1}-\frac{\delta_{1}t_{1}+\cdots+\delta_{s-1}t_{s-1}}{v-1} \\
 & = & \frac{(\frac{v}{2}-\left|k-\frac{v}{2}\right|)(\frac{v}{2}+\left|k-\frac{v}{2}\right|)
 -\sum_{i=1}^{s-1}\delta_{i}t_{i}}{(v-1)} \\
 & = & \frac{\frac{v^{2}}{4}-\left|k-\frac{v}{2}\right|^{2}-\sum_{i=1}^{s-1}\delta_{i}t_{i}}{(v-1)} \\
 & = & \frac{v^{2}}{4(v-1)}-\frac{\left|k-\frac{v}{2}\right|^{2}+\sum_{i=1}^{s-1}\delta_{i}t_{i}}{(v-1)} \\
 & = & \left\lfloor \frac{v^{2}}{4(v-1)} \right\rfloor. \quad (\text{by } (\ref{eq13}))
\end{eqnarray*} Together with Lemma \ref{le11} we have that equality holds. In fact, from the above inequality, we can see that the condition given in (\ref{eq13}) is both necessary and sufficient.
We now have a characterization of those subsets of cyclic groups for which equality holds in Lemma \ref{le11}.

 For the remainder of the paper we will denote the value $\left\lfloor \frac{v^{2}}{4(v-1)} \right\rfloor$ by $B_{v}$, and the peak-sidelobe distance of a binary matrix $R$ by $Q_{R}$, or, if it is clear from the context, simply by $Q$.
We will need the following lemma.
\begin{lemma}\label{le12} Let $R$ be a $v\times v$ binary circulant matrix. Then $A_{R}(1,0)=A_{R}(-1,0)=A_{R}(0,1)=A_{R}(0,-1)$.
\end{lemma}
\begin{proof} That $A_{R}(1,0)=A_{R}(-1,0)$ and $A_{R}(0,1)=A_{R}(0,-1)$ is clear. We show that $A_{R}(-1,0)=A_{R}(0,1)$. Note that for any $v\times v$ binary circulant matrix we must have $R_{i,j}=R_{v+1-j,v+1-i}$ for $1\leq i,j \leq v$. Thus, we have
\begin{eqnarray*} A_{R}(0,1) & = & \sum_{i=1}^{v}\sum_{j=1}^{v-1}R_{i,j}R_{i,j+1}+\sum_{i=1}^{v}R_{i,v}R_{i,v-1} \\
                             & = & \sum_{i=1}^{v}\sum_{j=2}^{v}R_{j,i}R_{j-1,i}+\sum_{i=1}^{v}R_{1,i}R_{0,i} \\
                             & = & \sum_{j=1}^{v}\sum_{i=1}^{v}R_{j,i}R_{j-1,i}\\
                             & = & A_{R}(-1,0). \end{eqnarray*}
\end{proof}
We are now ready to give an upper bound on the peak-sidelobe distance of a special matrix.
\begin{theorem}\label{th11} Let $R$ be a $v\times v$ special matrix with peak-sidelobe distance $Q$. Then \[
Q \leq (v+1)B_{v}+1. \]
\end{theorem}
\begin{proof} Let $R$ have defining set $D$ and largest difference level $\Lambda$. For any $i',\delta\in\mathbb{Z}$ we have \begin{eqnarray*}
\sum_{j=1}^{v}R_{i',j}R_{i',j+\delta} & \leq & \begin{cases} \Lambda-1, \text{ if } R_{i',1}=R_{i',v}=1, \\
                                                      \Lambda, \text{ otherwise,}\end{cases}\\
                                      & = & \sum_{j=1}^{v}R_{i',j}R_{i',j+1}.\end{eqnarray*} Using Lemma \ref{le12} it is easy to deduce that
$A_{R}(\tau_{1},\tau_{2})\leq A_{R}(\pm 1,0)=A_{R}(0,\pm 1)$ for all $\tau_{1},\tau_{2}$ not both zero. The number of values of $i'$ for which $R_{i',1}=R_{i',v}=1$ is equal to the number of pairs of consecutive residues in $D$ which, since $D$ is special, is $\Lambda$. Thus \[
A_{R}(0,1)=\sum_{i=1}^{v}\sum_{j=1}^{v}R_{i,j}R_{i,j+1}=\Lambda(\Lambda-1)+(v-\Lambda)\Lambda=(v-1)\Lambda,\] and \[
Q=A_{R}(0,0)-A_{R}(0,1)=vk-(v-1)\Lambda=v(k-\Lambda)+\Lambda.\]
Note that we have $2\leq k<v$ and $0\leq\Lambda\leq k$. We claim that $Q$ reaches its maximum value when $k-\Lambda$ reaches its maximum value. To see this, suppose that $k-\Lambda$ is at its maximum value and that there are $k^{*},\Lambda^{*}$, also satisfying $2\leq k^{*}<v$ and $0\leq \Lambda^{*} \leq k$, such that $k^{*}-\Lambda^{*}<k-\Lambda$ and $v(k-\Lambda)+\Lambda<v(k^{*}-\Lambda^{*})+\Lambda^{*}$. Then we have $v((k-\Lambda)-(k^{*}-\Lambda^{*}))<\Lambda^{*}-\Lambda$ whence $v<\Lambda^{*}-\Lambda$, which is impossible. This proves the claim. By Lemma \ref{le11}, the maximum possible value of $k-\Lambda$ is $B_{v}$. From Equation (\ref{eq13}) it is easy to deduce that $\Lambda\leq B_{v}+1$ whenever $k-\Lambda=B_{v}$. Thus $Q\leq (v+1)B_{v}+1$, and we are done.
\end{proof}
We will say that a $v\times v$ special matrix $R$ whose peak-sidelobe distance meets the bound given in Theorem \ref{th11} is {\it s-optimal}, and we say it is {\it near s-optimal} if it has a peak-sidelobe distance of $(v+1)B_{v}$.

\section{Constructions of s-Optimal Binary Matrices from Difference and Almost Difference Sets}\label{sec4}

In this section we will use cyclotomic classes and the group ring notation introduced in Section \ref{sec2}.  When convenient, we will denote the subset $\{b_{1},...,b_{k}\}\times S$ of an additive Abelian group $A \times B$ by $\{b_{1}\cdots,b_{k}\}S(x)$ where $S(x)$ is the polynomial in $\mathbb{Z}\left[A\right]$ corresponding to the subset $S$. We will only discuss those constructions which produce $v\times v$ binary circulant matrices whose peak-sidelobe distance is either $(v+1)B_{v}+1$ or $(v+1)B_{v}$, i.e. either s-optimal or near s-optimal. If we take the defining set $D$ of a binary circulant matrix $R$ to be a $(v,k,\lambda)$ difference set then, since $D$ only has one difference level, we have that the number of pairs of consecutive residues in $D$ is $\Lambda=\lambda$ and $R$ is special.
\par
Difference sets with parameters $(v,\frac{v-1}{2},\frac{v-1}{4})$ or $(v,\frac{v+1}{2},\frac{v+1}{4})$ are called {\it Paley-hadamard difference sets}.
Cyclic Paley-Hadamard difference sets, up to complementation (see Theorem \ref{th23}), include the following \cite{ACH}:\begin{enumerate}[(A)]
\item with parameters $(p,\frac{p+1}{2},\frac{p+1}{4})$, where $p\equiv3($mod $4)$ is prime, and the difference set is given by $D=D_{0}^{(2,p)}\cup \{0\}$,
\item with parameters $(2^{t}-1,2^{t-1},2^{t-2}+1)$, for descriptions see \cite{DIL},\cite{DD},\cite{GMW}, \cite{POTT} and \cite{XIANG},
\item with parameters $(v,\frac{v+1}{2},\frac{v+1}{4})$, where $v=p(p+2)$ and both $p$ and $p+2$ are primes. These are twin prime difference sets, and are defined by $\{(g,h)\in\mathbb{Z}_{p}\times\mathbb{Z}_{p+2}\mid g\neq0\neq h\text{ and } \chi(g)\chi(h)=-1\}\cup\{(0,h)\mid h\in \mathbb{Z}_{p+2}^{*}\}$ where $\chi(x)=1$ if $x$ is a nonzero square and $\chi(x)=-1$ otherwise \cite{JP},
\item with parameters $(p,\frac{p+1}{2},\frac{p+1}{4})$, where $p$ is a prime of the form $p=4s^2+27$. These are cyclotomic difference sets and are described in \cite{STO} as $D=D_{0}^{(6,p)}\cup D_{1}^{(6,p)}\cup D_{3}^{(6,p)}\cup \{0\}$.
    \end{enumerate}
We have the following construction.
\begin{theorem}\label{th20} Let $R$ be a $v\times v$ binary circulant matrix whose defining set is a Paley-Hadamard difference set of type (A), (B), (C) or (D). Then $R$ is near s-optimal.
\end{theorem}
\begin{proof}
It is a simple matter of using properties of the floor function to check that, in each of the cases (A), (B), (C) and (D), we have $k-\lambda=B_{v}$ and $\lambda=B_{v}$.
\end{proof}
If we take the defining set $D$ of a binary circlant matrix $R$ to be a $(v,k,\lambda,t)$ almost difference set, then $R$ is special only if the number of pairs of consecutive residues in $D$ is $\Lambda=\lambda+1$. In many cases it is difficult to know whether an almost difference set is special. In some cases, however, we can count the number of pairs of consecutive residues.

We will need the following lemma that can be found in \cite{D}.

\begin{lemma}\label{le21} If $q \equiv 1$ (mod 4) is a prime power then the cyclotomic numbers of order two are given by
\begin{eqnarray*}
(0,0) & = & \frac{q-5}{4},                         \\
(0,1) & = & (1,0) = (1,1) = \frac{q-1}{4}.
\end{eqnarray*}If $q \equiv 3$ (mod 4) then are given by
\begin{eqnarray*}
(0,1) & = & \frac{q+1}{4},                         \\
(0,0) & = & (1,0) = (1,1) = \frac{q-3}{4}.
\end{eqnarray*}
\end{lemma}
\begin{theorem}\label{th25} Let $p\equiv1($mod $4)$ be a prime and $R$ a $p\times p$ binary circulant matrix with defining set $D\subseteq\mathbb{Z}_{p}$. Then \begin{enumerate}
\item if $D=D_{0}^{(2,p)}\cup \{0\}$ then $R$ is s-optimal, and
\item if $D=D_{1}^{(2,p)}$ then $R$ is near s-optimal.\end{enumerate}
\end{theorem}
\begin{proof} We show only the first case as the second case can be shown in a similar way. To see that $D$ is special, note that, using Lemmas \ref{le2.0} and \ref{le21} we have \[
D(X)D(X^{-1})=\frac{p+1}{2}\cdot\mathbf{1}+\frac{p+3}{4}D_{0}^{(2,p)}(X)+\frac{p-1}{4}D_{1}^{(2,p)}(X).\]Since $1\in D_{0}^{(2,p)}$ we have that the number of pairs of consecutive residues in $D$ is $\Lambda=\frac{p+3}{4}$, whence $D$ is special. It is easy to show that $k-\Lambda=B_{p}$ and $\Lambda=B_{p}+1$ hold.
\end{proof}
Note that from the proof of Theorem \ref{th24} one could also deduce that the defining set is an almost difference set. The next construction also uses quadradic residues. We will need the following lemma.
\begin{lemma}\label{le25} {\rm \cite{YZ}} Let $p\equiv 3($mod $4)$ be a prime. Then \[
D=(\{0\}\times D_{0}^{(2,p)})\cup(\{1,2,3\}\times D_{1}^{(2,p)})\cup\{(0,0),(1,0),(3,0)\}\] is a $(4p,2p+1,p,p-1)$ almost difference set in $\mathbb{Z}_{4p}$.
\end{lemma}
\begin{theorem}\label{th26} Let $p\equiv 3($mod $4)$ be a prime. Let $R$ be the $4p\times 4p$ binary circulant matrix with defining set \[
D=(\{0\}\times D_{0}^{(2,p)})\cup(\{1,2,3\}\times D_{1}^{(2,p)})\cup\{(0,0),(1,0),(3,0)\}\subseteq\mathbb{Z}_{4p}.\] Then $R$ is s-optimal.
\end{theorem}
\begin{proof} Let $D_{i}$ denote $D_{i}^{(2,p)}$. To see that $D$ is special, using Lemmas \ref{le2.0} and \ref{le21} we have \small{\begin{eqnarray*}
D(X)D(X^{-1}) & = &
\left[ \{0\}(D_{0}(X)+\mathbf{1})+\{1,2,3\}(D_{1}(X)+\mathbf{1})\right]\left[\{0\}(D_{0}(X^{-1})+\mathbf{1})+\{1,2,3\}(D_{1}(X^{-1})+\mathbf{1})\right]\\
              & = & \{0\}(D_{0}(X)D_{0}(X^{-1})+D_{0}(X)+D_{1}(X)+\mathbf{1})\\
              &\, & \quad+\{1,2,3\}(D_{0}(X)D_{1}(X^{-1})+2D_{1}(X)D_{1}(X^{-1})+3D_{0}(X)+3D_{1}(X)+\mathbf{1})\\
              & = & \{0\}\left [(1,0)+(1,1)+1)D_{0}(X)+((1,1)+(1,0)+1)D_{1}(X)\right]\\
              &\, & \quad+\{1,2,3\}\left[((0,0)+(0,1)+2(1,1)+3)D_{0}(X)+((0,1)+(0,0)+2(1,0)+3)D_{1}(X)\right]\\
              &\, & \quad+(2p+1)\mathbf{1}.
\end{eqnarray*}}
If $\phi:\mathbb{Z}_{4p}\rightarrow\mathbb{Z}_{4}\times\mathbb{Z}_{p}$ is the map given by the Chinese Remainder Theorem, it is easy to see that $\phi^{-1}((1,1))=1$. Then since $1\in D_{0}$, we need only consider the coefficient of $\{1\}D_{0}$, which is $(0,0)+(0,1)+2(1,1)+3=p+1$. Thus, by Lemma \ref{le25}, the number of pairs of consecutive residues in $D$ is $\Lambda$, whence $D$ is special. It is easy to show that $k-\Lambda=B_{p}$ and $\Lambda=B_{p}+1$ hold.
\end{proof}
We next show a construction from cyclotomic classes of order four. We will use the following lemmas.
\begin{lemma}\label{le22} {\rm \cite{DHM}} Let $q=4f+1$ be a prime power with $f$ odd. The five distinct cyclotomic numbers are
\begin{eqnarray*}
(0,0) & = & (2,2) = (2,0) = \frac{q-7+2x}{16}    \\
(0,1) & = & (1,3) = (3,2) = \frac{q+1+2x-8y}{16} \\
(1,2) & = & (0,3) = (3,1) = \frac{q+1+2x+8y}{16} \\
(0,2) & = & \frac{q+1-6x}{16}                    \\
\text{all others} & = & \frac{q-3-2x}{16}
\end{eqnarray*} where $q=x^{2}+4y^{2}$ for $x,y \in \mathbb{Z}$ with $x \equiv 1$ (mod 4). Here, $y$ is two-valued depending on the choice of the primitive root $\alpha$ defining the cyclotomic classes (see page 400 of {\rm \cite{D}}).
\end{lemma}
\begin{lemma}\label{le23} {\rm \cite{DHM}} Let $p=4f+1=x^{2}+4y^{2}$ be a prime with $f$ odd, $x\equiv1($mod $4)$ and $y=\pm 1$. Then $D_{i}^{(4,p)}\cup D_{i+1}^{(4,p)}$ is a $(p,\frac{p-1}{2},\frac{p-5}{4},\frac{p-1}{2})$ almost difference set in $\mathbb{Z}_{p}$.
\end{lemma}
\begin{theorem}\label{th27} Let $p=4f+1=x^{2}+4y^{2}$ be a prime with $f$ odd, $x\equiv1($mod $4)$, and primitive element chosen so that $y=-1$. Let $R$ be a $p\times p$ binary circulant matrix with defining set $D\subseteq\mathbb{Z}_{p}$. Then \begin{enumerate}
\item if $D=D_{i}^{(4,p)}\cup D_{i+1}^{(4,p)}\cup \{0\}$ then $R$ is s-optimal, and
\item if $D=D_{i}^{(4,p)}\cup D_{i+1}^{(4,p)}$ then $R$ is near s-optimal.\end{enumerate}
\end{theorem}
\begin{proof} We show only the first case as the second case can be shown in a similar way. Let $D_{i}$ denote $D_{i}^{(4,p)}$. To see that $D$ is special, using Lemma \ref{le2.0} we have \begin{eqnarray*}
D(X)D(X^{-1}) & = & (D_{i}\cup D_{i+1}+\mathbf{1})(X)(D_{i}\cup D_{i+1}+\mathbf{1})(X^{-1}) \\
              & = & (D_{i}\cup D_{i+1}+\mathbf{1})(X)(D_{i+2}\cup D_{i+3}+\mathbf{1})(X) \\
              & = & \frac{p+1}{2}\cdot\mathbf{1}+\sum_{k=0}^{3}\left[(2,k-i)+(3,k-i)+(1,k-i-1)+(2,k-i-1)+1\right] D_{k}(X). \end{eqnarray*}
Using properties of the cyclotomic numbers together with Lemma \ref{le22}, we have that the coefficient of $D_{0}(X)$ is \[ \begin{cases} \frac{4p+4-8y}{16}, \text{ if } i=0 \text{ or } 2, \\
                                  \frac{4p+4+8y}{16}, \text{ if } i=1 \text{ or } 3.\end{cases}\]By Lemma \ref{le23} and Theorem \ref{th24} we know that $\Lambda=\frac{p+3}{4}$. Since $1\in D_{0}$ we have that the number of pairs of consecutive residues in $D$ is $\Lambda$, whence $D$ is special. It is easy to show that $k-\Lambda=B_{p}$ and $\Lambda=B_{p}+1$ hold.
\end{proof}

\section{Contruction of Binary Matrices with Good Peak-Sidelobe Distances}\label{sec5}
In this section we give a construction of binary matrices with good peak-sidelobe distances from s-optimal matrices by following the strategy mentioned in Section \ref{ssec1.1}. Let $R$ be a $(v-2)\times(v-2)$ s-optimal matrix with defining set $D$ of cardinality $k$. Then, by Theorem \ref{th11}, $R$ has peak-sidelobe distance $(v-1)B_{v-2}+1$. Let $R'$ be the $v\times v$ matrix obtained by adjoining a border of $1$s to $R$. i.e. \[
R'=\left[\begin{array}{ccc}
1     &\cdots&1     \\
\vdots&R     &\vdots\\
1     &\cdots&1     \end{array}\right].\] Since the nearest sidelobe of $R$ occurs in $\{A_{R}(\pm1,0),A_{R}(0,\pm1)\}$, and $R$ is circulant, it is easy to see that the nearest sidelobe of $R'$ must occur in $\{A_{R'}(\pm1,0),A_{R'}(0,\pm1)\}$ and that the peak-sidelobe distance of $R'$ is given by \[
Q_{R'}=A_{R'}(0,0)-A_{R'}(\pm1,0)=A_{R'}(0,0)-A_{R'}(0,\pm1)=(v-1)B_{v-2}+2(v-k)+1.\] Now fix $a,b\in\{1...v\}$ and define $S_{aj}^{b}=\{(a,j)\mid 2\leq j\leq v-1,R_{bj}'=1\}$ and $S_{ia}^{b}=\{(j,a)\mid 2\leq i\leq v-1,R_{ib}'=1\}$. Then the exterior entries of the matrix $R'$ having adjacent interior entries equal to $1$ can be represented by the sets $S_{1j}^{2},S_{vj}^{v-1},S_{i1}^{2}$ and $S_{iv}^{v-1}$. Now modify $R'$ by choosing one entry from each of $S_{1j}^{2},S_{vj}^{v-1},S_{i1}^{2}$ and $S_{iv}^{v-1}$, and changing it to $0$. Let $I$ denote the set of indices for the four chosen entries, and $R_{I}'$ denote the resulting matrix.
It is straightforward to show that the nearest sidelobe of $R_{I}'$ occurs in $\{A_{R_{I}'}(\pm1,0),A_{R_{I}'}(0,\pm1)\}$ and that $R_{I}'$ has peak-sidelobe distance given by $Q_{R_{I}'}=(v-1)B_{v-2}+2(v-k)+3$. Similarly, if $R$ is near s-optimal then $R_{I}'$ has peak-sidelobe distance $Q_{R_{I}'}=(v-1)B_{v-2}+2(v-k)+2$.
We thus have the following.
\begin{theorem}\label{th24} Let $R$ be a $(v-2)\times(v-2)$ special matrix with defining set $D$ of cardinality $k$, and let $I$ contain exactly one index from each of $S_{1j}^{2},S_{vj}^{v-1},S_{i1}^{2}$ and $S_{iv}^{v-1}$. Then $R_{I}'$ is a $v\times v$ binary matrix with peak-sidelobe distance \[
Q_{R_{I}'}=\begin{cases} (v-1)B_{v-2}+2(v-k)+3, \text{ if } R \text{ is s-optimal},\\
                         (v-1)B_{v-2}+2(v-k)+2, \text{ if } R \text{ is near s-optimal}.\end{cases}\]
\end{theorem}
\begin{example}\label{ex1} Let $D$ be the set of quadratic nonresidues in $\mathbb{Z}_{5}$ and $R$ the binary circulant matrix with defining set $D$. Then $S_{1j}^{2}=\{(1,4),(1,5)\},S_{vj}^{v-1}=\{(7,3),(7,4)\},S_{i1}^{2}=\{(4,1),(5,1)\}$ and $S_{iv}^{v-1}=\{(3,7),(4,7)\}$. Take $I=\{(1,4),(7,4),(4,1),(4,7)\}$. Then we have \[
 R=\left[\begin{array}{ccccc}
0&0&1&1&0\\
0&0&0&1&1\\
1&0&0&0&1\\
1&1&0&0&0\\
0&1&1&0&0\end{array}\right]\text{ and }
R_{I}'=\left[\begin{array}{ccccccc}
1&1&1&0&1&1&1\\
1&0&0&1&1&0&1\\
1&0&0&0&1&1&1\\
0&1&0&0&0&1&0\\
1&1&1&0&0&0&1\\
1&0&1&1&0&0&1\\
1&1&1&0&1&1&1\end{array}\right].\] By Theorem \ref{th20}, $R$ is near s-optimal, and so by Theorem \ref{th24} we have that the peak-sidelobe distance of $R_{I}'$ is $Q_{R_{I}'}=(v-1)B_{v-2}+2(v-k)+2=18$. According to \cite{SKIR}, an optimal $7\times 7$ binary matrix has peak-sidelobe distance $19$.
\end{example}
\begin{example}\label{ex2} Let $D$ be the set of quadratic residues in $\mathbb{Z}_{7}$ and $R$ the binary circulant matrix with defining set $D\cup\{0\}$. Then $S_{1j}^{2}=\{(1,2),(1,3),(1,4),(1,6)\},S_{vj}^{v-1}=\{(9,2),(9,5),(9,4),(9,6)\},S_{i1}^{2}=\{(2,1),(5,1),(7,1),(8,1)\}$ and $S_{iv}^{v-1}=\{(2,9),(5,9),(7,9),(8,9)\}$. Take $I=\{(1,4),(9,5),(5,1),(6,9)\}$. Then we have \[ R=\left[\begin{array}{ccccccc}
1&1&1&0&1&0&0\\
0&1&1&1&0&1&0\\
0&0&1&1&1&0&1\\
1&0&0&1&1&1&0\\
0&1&0&0&1&1&1\\
1&0&1&0&0&1&1\\
1&1&0&1&0&0&1\end{array}\right]\text{ and }
R_{I}'=\left[\begin{array}{ccccccccc}
1&1&1&0&1&1&1&1&1\\
1&1&1&1&0&1&0&0&1\\
1&0&1&1&1&0&1&0&1\\
1&0&0&1&1&1&0&1&1\\
0&1&0&0&1&1&1&0&1\\
1&0&1&0&0&1&1&1&0\\
1&1&0&1&0&0&1&1&1\\
1&1&1&0&1&0&0&1&1\\
1&1&1&1&0&1&1&1&1\end{array}\right].\] By Theorem \ref{th20}, $R$ is near s-optimal, and so by Theorem \ref{th24} we have that the peak-sidelobe distance of $R_{I}'$ is $Q_{R_{I}'}=(v-1)B_{v-2}+2(v-k)+2=28$. According to \cite{SKIR}, the best peak-sidelobe distance for a $9\times 9$ binary matrix is $29$.
\end{example}
\begin{remark} Note that in the two examples given above, although we have chosen the set $I$ so that the resulting matrix $R_{I}'$ is symmetric, a different choice of the set $I$ will have no effect on the peak-sidelobe distance as long as the condition in the statement of Theorem \ref{th20} is satisfied.
\end{remark}
\begin{remark} Also note that a square binary matrix constructed via Theorem \ref{th24} from an s-optimal interior matrix does not necessarily have a greater peak-sidelobe distance than one constructed from a near s-optimal interior matrix. This can easily be seen by comparing the matrices that result from applying Theorem \ref{th24} to interior matrices constructed using parts (1) and (2) of either of Theorems \ref{th25} and \ref{th27}.
\end{remark}
We have computed the best peak-sidelobe distances of square binary matrices constructed via Theorem \ref{th24} of orders between $7$ and $19$ and included then in Table \ref{ta1}. We have also computed explicitly the paramters and peak-sidelobe distances for the infinite families of square binary matrices constructed in this paper and included them Table \ref{ta2}.
\section{Concluding Remarks}\label{sec6}
We have shown how difference sets and almost difference sets can be used to construct binary matrices with peak-sidelobe distances that are good in the sense that, for small dimension (as illustrated in Examples \ref{ex1} and \ref{ex2}), are very close to being optimal. Before this paper, there have been no results on deterministically constructing such families of binary matrices and, so far, have only been constructed via exhaustive computer searches. We have formulated a tight general upper bound on the peak-sidelobe distance of a certain class of circulant matrices whose defining sets are difference sets and almost difference sets. By using circulant matrices whose aperiodic autocorrelations meet this upper bound, we were able to construct square binary matrices with good aperiodic autocorrelation properties. We leave the reader with the following open problems: 1) Formulate a tight general upper bound on the peak-sidelobe distance a square binary matrix or, improve on the upper bound formulated by Skirlo et al. in \cite{SKIR}, 2) Find new ways of constructing square binary matrices with good aperiodic autocorrelation properties.

\begin{center}
\captionof{table}{{\small Table of peak-sidelobe distances of square binary matrices constructed via Theorem \ref{th24} of orders between $7$ and $19$.}}\label{ta1}
\begin{tabular}{|>{\small}c|>{\small}c|>{\small}c|>{\small}c|}
\multicolumn{4}{l}{{\footnotesize Note: 'DS' and 'ADS' refer to 'difference set' and 'almost difference set' respectively.}} \\
\multicolumn{4}{l}{{\footnotesize Note: Orders given with a '*' refer to those for which previous results exist and for which comparisons are made }}\\
\multicolumn{4}{l}{{\footnotesize \quad \quad \quad in Examples \ref{ex1} and \ref{ex2}.}} \\
  \hline
   Order $M$ & \shortstack{Ref. to construction\\ of $(M-2)\times(M-2)$ interior}& \shortstack{Param. of DS or ADS used \\as defining set $D$ of $R$} &\shortstack{Peak-sidelobe distance $Q_{R_{I}'}$ \\of $M\times M$ matrix $R_{I}'$} \\
  \hline\hline
  $7^{*}$&Theorem \ref{th25} part(2)&$(5,2,0,2)$-ADS&$18$ \\\hline
 $9^{*}$&Theorem \ref{th20}&$(7,4,2)$-DS&$28$ \\\hline
 $13$&Theorem \ref{th20}&$(11,6,3)$-DS&$52$ \\\hline
 $14$&Theorem \ref{th26}&$(12,7,3,2)$-ADS&$56$ \\\hline
 $15$&\shortstack{Theorem \ref{th25} part(2),\\Theorem \ref{th27} part(2)} &$(13,6,2,6)$-ADS&$62$ \\\hline
 $17$&Theorem \ref{th20}&$(15,8,4)$-DS&$84$\\\hline
 $19$&\shortstack{Theorem \ref{th25} part(2),\\Theorem \ref{th27} part(2)} &$(17,8,3,8)$-ADS&$96$ \\\hline

\end{tabular}
\end{center}
\newpage
\begin{center}
\captionof{table}{{\small Table of parameters of families of square binary matrices constructed via Theorem \ref{th24}.}}\label{ta2}
\begin{tabular}{|>{\small}c|>{\small}c|>{\small}c|}
\multicolumn{3}{l}{{\footnotesize Note: 'DS' and 'ADS' refer to 'difference set' and 'almost difference set' respectively.}} \\
  \hline
   \shortstack{Ref. to construction of \\$(M-2)\times(M-2)$ interior} & \shortstack{Param. of DS or ADS used \\as defining set $D$ of $R$} &\shortstack{Peak-sidelobe distance $Q_{R_{I}'}$ of \\$M\times M$ matrix $R_{I}'$} \\
  \hline\hline
Theorem \ref{th20}&\shortstack{$(v,\frac{v+1}{2},\frac{v+1}{4})$-DS of \\Paley-Hadamard type $(A),(B),(C)$ or $(D)$}&$(v+1)\left(\lfloor\frac{v^{2}}{4(v-1)}\rfloor+1\right)+4$ \\\hline
\shortstack{Theorem \ref{th25} part(1),\\Theorem \ref{th27} part(1)} &\shortstack{$(p,\frac{p+1}{2},\frac{p-1}{4},\frac{p-1}{2})$-ADS of \\where $p\equiv 1($mod $4)$ is prime}&$(p+1)\left(\lfloor\frac{p^{2}}{4(p-1)}\rfloor+1\right)+5$ \\\hline
\shortstack{Theorem \ref{th25} part(2),\\Theorem \ref{th27} part(2)} &\shortstack{$(p,\frac{p+1}{2},\frac{p-1}{4},\frac{p-1}{2})$-ADS of \\where $p\equiv 1($mod $4)$ is prime}&$(p+1)\left(\lfloor\frac{p^{2}}{4(p-1)}\rfloor+1\right)+6$ \\\hline
Theorem \ref{th26}&\shortstack{$(4p+1,2p+1,p,p-1)$-ADS\\where $p\equiv 3($mod $4)$ is prime}&$(4p+1)\left(\lfloor\frac{4p^{2}}{(4p-1)}\rfloor+1\right)+4$ \\\hline

\end{tabular}
\end{center}

\bibliographystyle{plain}
\bibliography{myref}

\end{document}